\documentclass[12pt]{article}

\usepackage{amsfonts,amsthm}
\usepackage{color}

\def\red#1{\begin{color}{black}#1\end{color}}

\newtheorem{theorem}{Theorem}
\newtheorem{problem}{Problem}

\begin{document}
 \title{On the maximum order complexity of subsequences of the Thue-Morse and Rudin-Shapiro sequence along squares}

 \author{Zhimin Sun$^1$ and Arne Winterhof$^2$\\ $^1$   Faculty of Mathematics and Statistics,\\ Hubei Key Laboratory of Applied Mathematics,\\ Hubei
University, Wuhan, 430062,
 China\\
 e-mail: zmsun@hubu.edu.cn\\
 $^2$ Johann Radon Institute for Computational and Applied
Mathematics,\\
Austrian Academy of Sciences,\\
Altenberger Stra{\ss}e 69, A-4040 Linz, Austria\\
e-mail: arne.winterhof@oeaw.ac.at}

\date{}

 \maketitle
 
 \begin{abstract}
 	Automatic sequences such as the Thue-Morse sequence and the Rudin-Shapiro sequence are highly predictable and thus not suitable in cryptography. 
 	In particular, they have small expansion complexity.
 	However, they still have a large maximum order complexity.
 	
 	Certain subsequences of automatic sequences are not automatic anymore and may be attractive candidates for applications in cryptography. In this paper we show that subsequences along the squares 
 	of certain pattern sequences including the Thue-Morse sequence and the Rudin-Shapiro sequence have also large maximum order complexity but do not suffer a small expansion complexity anymore.
 \end{abstract}

Keywords. Thue-Morse sequence, Rudin-Shapiro sequence, automatic sequence, maximum order complexity, measures of pseudorandomness \\

MSC 2010. 11B85, 11K45
 
\section{Introduction}
 For a positive integer $N$ and a sequence ${\cal S}=(s_i)_{i=0}^\infty$ over the finite field~${\mathbb F}_2$ of two elements with $(s_0,\ldots,s_{N-2})\ne (a,\ldots,a)$, $a\in \{0,1\}$,
 the {\em $N$th maximum order complexity} $M({\cal S},N)$ (or {\em $N$th nonlinear complexity}) is the smallest positive integer $M$
such that there is a polynomial $f(x_1,\ldots,x_M)\in {\mathbb F}_2[x_1,\ldots,x_M]$
with
$$s_{i+M}=f(s_i,s_{i+1},\ldots,s_{i+M-1}),\quad 0\le i\le N-M-1,$$
see \cite{ja89,ja91}. If $s_i=a$ for $i=0,\ldots,N-2$, we define $M({\cal S},N)=0$ if $s_{N-1}=a$ and $M({\cal S},N)=N-1$ if $s_{N-1}\ne a$.
A sequence with small $N$th maximum order complexity (for sufficiently large $N$) is predictable and thus unsuitable in cryptography.
However, there are predictable sequences with large $N$th maximum order complexity and further quality measures for cryptographic sequences have to be studied.

Diem \cite{di12} introduced the expansion complexity of the sequence ${\cal S}$ as follows. We define the {\em generating function} $G(x)$ of ${\cal S}$ by
$$G(x)=\sum_{i=0}^\infty s_i x^i,$$
viewed as a formal power series over ${\mathbb F}_2$. (Note the change by the factor $x$ compared to the definition in~\cite{di12}.)
For a positive integer $N$, the {\em $N$th expansion complexity} $E({\cal S},N)$ is $0$ if $s_0=\ldots=s_{N-1}=0$ and otherwise the least total degree
of any nonzero polynomial $h(x,y)\in {\mathbb F}_2[x,y]$ with
$$h(x,G(x))\equiv 0 \bmod x^N.$$
A sequence with small $N$th expansion complexity is predictable.

Automatic sequences such as the Thue-Morse sequence and the Rudin-Shapiro sequence have a large $N$th maximum order complexity of order of magnitude $N$, see \cite{suwi}.
However, by Christol's theorem \cite{ch} they are characterized by 
$$\sup_{N\ge 1} E({\cal S},N)<\infty,$$
see also \cite[Theorem~12.2.5]{alsh}.

For example, the {\em Thue-Morse sequence} ${\cal T}=(t_i)_{i=0}^\infty$ over ${\mathbb F}_2$ is defined by
$$t_i=\left\{ \begin{array}{cl} t_{i/2} & \mbox{if $i$ is even},\\
t_{(i-1)/2}+1 & \mbox{if $i$ is odd},
\end{array}\right.\quad i=1,2,\ldots
$$
with initial value $t_0=0$.  
An explicit formula for $M({\cal T},N)$ is given in \cite[Theorem 1]{suwi}. In particular, it satisfies 
$$M({\cal T},N)\ge \frac{N}{5}+1, \quad N\ge 4.$$
However, taking
$$h(x,y)=(x+1)^3 y^2+(x+1)^2 y+x,$$
its generating function $G(x)$ satisfies $h(x,G(x))=0$
and thus
$$E({\cal T},N)\le 5,\quad N=1,2,\ldots$$
Hence, despite of a large $N$th maximum order complexity, the Thue-Morse sequence is highly predictable.
\red{Other indicators for its predictability are a linear subword complexity, see \cite[Exercise 10.11.10]{alsh} or \cite{br,luva}, and a large correlation measure of order $2$ \cite{masa}.}

Subsequences of automatic sequences may be not automatic anymore and can look much more random. 
For example, the subsequence of the Thue-Morse sequence along squares ${\cal T}'=(t_{i^2})_{i=0}^\infty$ is not automatic by \cite[Theorem 6.10.1]{alsh}, that is,
$$\sup_{N\ge 1} E({\cal T}',N)=\infty,$$
\red{it has the largest possible subword complexity \cite{mo07} and is even} normal \cite{drmari17}.
In Section~\ref{Thue} we prove a lower bound on $M({\cal T}',N)$ of order of magnitude $N^{1/2}$, which indicates that ${\cal T}'$ is rather unpredictable. 

More generally, for a positive integer $k$ we study subsequences along the squares of the {\em pattern sequences} ${\cal P}_{k}=(p_n)_{n=0}^\infty$ over ${\mathbb F}_2$ defined by
$$p_n\equiv s_k(n) \bmod 2,$$
where $P_k=11\ldots 1 \in \mathbb{F}_{2}^{k}$ is the all $1$ pattern of length $k$ and $s_k(n)$ 
is the number of occurrences of $P_k$ in the binary representation of $n$.
For $k=1$ we get the Thue-Morse sequence and for $k=2$ the {\em Rudin-Shapiro sequence}.
In Section~\ref{pattern} for $k\ge 2$ we prove a lower bound on the maximum order complexity of ${\cal P}_k'$ of order of magnitude $N^{1/2}$.
Note that the proof is slightly different than for $k=1$.

\red{We finish this paper with a list of open problems in Section~\ref{open}.}

\section{The Thue-Morse sequence along squares}
\label{Thue}

 \begin{theorem}\label{th1}
 Let ${\cal T}^{'} =(t_{i^{2}})_{i=0}^\infty$ be the subsequence of the Thue-Morse sequence along squares.
 Then the $N$th maximum order complexity of ${\cal T}'$ satisfies
   $$ M({\cal T}^{'} ,N)\geq \sqrt{\frac{2N}{5}},\quad N\ge 21.$$
\end{theorem}

 \begin{proof}
Let $\ell\ge 2$ be the integer defined by 
\begin{equation}\label{N} 5\cdot 2^\ell<N\le 5\cdot 2^{\ell+1}
\end{equation}
and note that the Thue-Morse sequence satisfies
$$t_n\equiv s_1(n)\bmod 2,\quad n=0,1,\ldots,$$
where $s_1(n)$ denotes the number of $n_i=1$ in the binary expansion of $n$, that is,  
$$n=\sum_{i=0}^\infty n_i2^i\quad\mbox{with}\quad n_i\in \{0,1\}.$$ 
(Note that only finitely many $n_i$ are nonzero.)

For $i=0,1,\ldots,\left\lfloor \sqrt{2^{\ell+2}-1}\right\rfloor$ we have (since $\ell\ge 2$)
$$t_{(i+2^{\ell+1})^2}\equiv s_1(i^2+i2^{\ell +2}+2^{2\ell+2})\equiv s_1(i^2)+s_1(i)+1\bmod 2$$
and 
$$t_{(i+2^{\ell+2})^2}\equiv s_1(i^2+i2^{\ell +3}+2^{2\ell+4})\equiv s_1(i^2)+s_1(i)+1\bmod 2$$
and thus 
\begin{equation}\label{ti} t_{(i+2^{\ell+1})^2}=t_{(i+2^{\ell+2})^2},\quad i=0,1,\ldots,\left\lfloor \sqrt{2^{\ell+2}-1}\right\rfloor.
\end{equation}
Moreover, we have 
$$s_1((2^\ell+2^{\ell+1})^2)=s_1(2^{2\ell}+2^{2\ell+3})=2$$
but
$$s_1((2^\ell+2^{\ell+2})^2)=s_1(2^{2\ell}+2^{2\ell+3}+2^{2\ell+4})=3.$$
Hence,
\begin{equation}\label{contra}
t_{(2^{\ell}+2^{\ell+1})^2}\ne t_{(2^\ell+2^{\ell+2})^2}.
\end{equation}
Now assume 
$$M=M({\cal T}',N)\le\left\lfloor \sqrt{2^{\ell+2}-1}\right\rfloor+1,$$
that is, there is a polynomial $f(x_1,\ldots,x_M)$ in $M$ variables with 
\begin{equation}\label{tM} t_{(\red{j}+M)^2}=f(t_{\red{j}^2},\ldots,t_{(\red{j}+M-1)^2}),\quad \red{j}=0,1,\ldots,N-M-1.
\end{equation}
\red{Note that for $0\le k\le N-M-1$ the values of $t_{(k+M)^2},t_{(k+M+1)^2},\ldots,t_{(N-1)^2}$ are uniquely determined by the values of $t_{k^2},\ldots,t_{(k+M-1)^2}$
and by applying successively the recurrence $(\ref{tM})$ for $j=k,\ldots,N-M-1$. In particular, if 
$$(t_{k_1^2},\ldots,t_{(k_1+M-1)^2})=(t_{k_2^2},\ldots,t_{(k_2+M-1)^2})$$ 
for some $k_1$ and $k_2$ with $0\le k_1<k_2\le N-M-1$, we get also
$$(t_{(k_1+M)^2},\ldots,t_{(k_1+N-k_2-1)^2})=(t_{(k_2+M)^2},\ldots,t_{(N-1)^2}).$$
Taking $k_1=2^{\ell+1}$ and $k_2=2^{\ell+2}$ we get from 
$(\ref{ti})$:
%\begin{eqnarray*} 
$$(t_{(2^{\ell+1}+M)^2},\ldots,t_{(N-2^{\ell+1}-1)^2})
%&=&f(t_{(i+2^{\ell+1})^2},\ldots,t_{(i+2^{\ell+1}+M-1)^2})\\
%&=&f(t_{(i+2^{\ell+2})^2},\ldots,t_{(i+2^{\ell+2}+M-1)^2})
=(t_{(2^{\ell+2}+M)^2},\ldots,t_{(N-1)^2}).$$
%\end{eqnarray*}
}
Since $N-1\ge 2^\ell+2^{\ell+2}$ (by the lower bound in $(\ref{N})$) \red{and $M\le 2^\ell$ this includes 
$$t_{(2^\ell+2^{\ell+1})^2}=t_{(2^\ell+2^{\ell+2})^2}$$ 
which} contradicts $(\ref{contra})$ and we get (using the upper bound in $(\ref{N})$)
$$M({\cal T}',N)\ge \left\lfloor \sqrt{2^{\ell+2}-1}\right\rfloor+2\ge \sqrt{\frac{2N}{5}},$$
which completes the proof.
\end{proof}

Remarks. 
\begin{enumerate}
\item Since the $N$th linear complexity is lower bounded by the $N$th maximum order complexity, this result shows that an attack via the Berlekamp-Massey algorithm fails for sufficiently large $N$. 
\item Our experimental results support the conjecture that Theorem~\ref{th1} is (up to the constant) best possible, that is, $M({\cal T}',N)$ is of order of magnitude $\sqrt{N}$.
\item \red{Our lower bound is strong enough to guarantee that ${\cal T}'$ is not vulnerable under any known algorithm that calculates a shortest recurrence relation. This is even true if we consider the simpler 
problem of finding a shortest linear recurrence, see Remark 1. However, it does not guarantee that there
is no other efficient way to attack our sequence although we are not aware of any such possible attack. Hence it is still important to study further features of this sequence such as its expansion complexity
or its correlation measure of order $k$. For further discussions about predictability and measures of pseudorandomness we refer to the surveys \cite{homewi,towi}.}
\item Further experiments indicate that also the $N$th expansion complexity of ${\cal T}'$ is quite large, that is, we believe that its order of magnitude is close to the best possible order $N^{1/2}$.
(We have $E({\cal S},N)\le \sqrt{2N}$ for any sequence ${\cal S}$ by \cite[Theorem~1]{gomeni}.)
\end{enumerate}

\section{Pattern sequences along squares for $k\ge 2$}
\label{pattern}

 \begin{theorem}\label{th2}
	For $k\ge 2$ let ${\cal P}_k^{'} =(p_{i^{2}})_{i=0}^\infty$ be the subsequence of the pattern sequence ${\cal P}_k$ along the squares.
	Then the $N$th maximum order complexity of~${\cal P}_k^{'}$  satisfies
	$$ M({\cal P}_k^{'} ,N)\geq \left(\frac{N}{8}\right)^{1/2},\quad N\ge 2^{2k+2}.$$
\end{theorem}

\begin{proof}
Let $\ell$ be the positive integer defined by 
\begin{equation}\label{Nrange} 2^{2k+\ell+1}\le N< 2^{2k+\ell+2}.
\end{equation}

For $0\le i\le \left\lfloor \sqrt{2^{\ell+2k-1}-1}\right\rfloor$ we have
$$s_k((i+2^{\ell+2k-1})^2)=s_k(i^2+i2^{\ell+2k}+2^{2\ell+4k-2})=s_k(i^2)+s_k(i)$$
as well as
$$s_k((i+2^{\ell+2k})^2)=s_k(i^2)+s_k(i).$$
\red{Thus
\begin{equation}\label{sk} p_{(i+2^{\ell+2k-1})^2}=p_{(i+2^{\ell+2k})^2},\quad i=0,1,\ldots, \left\lfloor \sqrt{2^{\ell+2k-1}-1}\right\rfloor.
\end{equation}}

For $k=2$ we have
$$s_2((2^{\ell+2}+2^{\ell+3})^2)=s_2(1+2^{3})=0$$
but
$$s_2((2^{\ell+2}+2^{\ell+4})^2)=s_2(1+2^3+2^4)=1$$
\red{and thus
\begin{equation}\label{k=2} p_{(2^{\ell+2}+2^{\ell+3})^2}\ne p_{(2^{\ell+2}+2^{\ell+4})^2}.
\end{equation}}

For even $k>2$ we have
$$s_k(((2^k-1)2^{\ell}+2^{2k-1+\ell})^2)=s_k(1+2^{3k}-2^{k+1}+2^{4k-2})=k\equiv 0\bmod 2$$
but
$$s_k(((2^k-1)2^{\ell}+2^{2k+\ell})^2)=s_k(1+(2^{2k}-2^{k+1})+(2^{3k+1}-2^{2k+1})+2^{4k})=1$$
\red{and thus
\begin{equation}\label{keven} p_{((2^k-1)2^{\ell}+2^{2k-1+\ell})^2}\ne p_{((2^k-1)2^{\ell}+2^{2k+\ell})^2}.
\end{equation}
}

For $k=3$ we have
$$s_3((7\cdot 2^{\ell+3}+2^{\ell+5})^2)=s_3(1+2^7-2^3)=2\equiv 0\bmod 2$$
but
$$s_3((7\cdot 2^{\ell+3}+2^{\ell+6})^2)=s_3(1+2^8-2^5)=1$$
\red{and thus
\begin{equation}\label{k=3}
 p_{(7\cdot 2^{\ell+3}+2^{\ell+5})^2}\ne p_{(7\cdot 2^{\ell+3}+2^{\ell+6})^2}.
\end{equation}}

For odd $k>3$ we have
$$s_k(((2^{k-1}-1)2^{\ell\red{+2}}+2^{2k-\red{1}+\ell})^2)=s_k(1+(2^{3k-3}-2^k)+2^{4k-6})=k-2\equiv 1\bmod 2$$
but
\begin{eqnarray*}
&&s_k(((2^{k-1}-1)2^{\ell\red{+2}}+2^{2k\red{+\ell}})^2)\\&=&s_k(1+(2^{2k-2}-2^k)+(2^{3k-2}-2^{2k-1})+2^{4k-4})\\
&=&0
\end{eqnarray*}
\red{and thus
\begin{equation}\label{kodd}
 p_{((2^{k-1}-1)2^{\ell\red{+2}}+2^{2k-\red{1}+\ell})^2}\ne p_{((2^{k-1}-1)2^{\ell\red{+2}}+2^{2k\red{+\ell}})^2}.
\end{equation}
}

Now the result follows the same way as Theorem~\ref{th1} \red{as follows}.

\red{Assume $M=M({\cal P}'_k,N)\le \left\lfloor \sqrt{2^{\ell+2k-1}-1}\right\rfloor+1$ and thus there is a recurrence of order $M$ which successively continues $(\ref{sk})$ to get
$$p_{(i+2^{\ell+2k-1})^2}=p_{(i+2^{\ell+2k})^2},\quad i=0,1,\ldots,N-2^{\ell+2k}-1.$$
Choosing 
$$i=\left\{\begin{array}{ll} 2^{\ell+2}, & k=2,\\
          (2^k-1)2^\ell,  & k>2 \mbox{ and $k$ even},\\
       7\cdot 2^{\ell+3}, & k=3,\\
          ((2^{k-1}-1)2^{\ell\red{+2}}, & k>3 \mbox{ and $k$ odd},
           \end{array}\right.$$
we get a contradiction to $(\ref{k=2})$, $(\ref{keven})$, $(\ref{k=3})$ or $(\ref{kodd})$, respectively.
Hence, 
$$M\ge \left\lfloor \sqrt{2^{\ell+2k-1}-1}\right\rfloor+2\ge \left(\frac{N}{8}\right)^{1/2}$$
by $(\ref{Nrange})$.
}
\end{proof}

\red{
\section{Open problems}\label{open}
}\red{
\begin{problem}\cite[Conjecture 1]{drmari17}
  Show that the subsequences of the Thue-Morse sequence (pattern sequence) along any polynomial of degree $d\ge 2$ are normal.
\end{problem}
This problem may be out of reach and we state some weaker problems.
}\red{
It is known that the subword complexity is maximal if $d=2$ \cite{mo07}. For $d\ge 3$ a lower bound on the subword complexity is given in \cite{mo07}, as well.
\begin{problem}\cite[Open Question 4]{mo07}
 Show that the subword complexity of the subsequence of the Thue-Morse sequence along any polynomial of degree $d\ge 3$ is maximal.
\end{problem}
}\red{
\begin{problem}\cite[above Conjecture 1]{drmari17}
 Determine the frequency of $0$ and $1$ in the subsequence of the Thue-Morse sequence along any polynomial of degree $d\ge 3$.
\end{problem}
}\red{
\begin{problem}
 Extend Theorems~\ref{th1} and \ref{th2} to any polynomial of degree $d\ge 2$.
\end{problem}
}\red{
\begin{problem} Prove upper bounds on the correlation measure of order $k$ for subsequences of the Thue-Morse sequence along squares (polynomials). 
\end{problem}
}\red{
\begin{problem}
 Prove lower bounds on the expansion complexity of the Thue-Morse sequence along squares (polynomials).
\end{problem}
}

\section*{Acknowledgments}
\red{We thank the referees for their careful reading and their valuable remarks.}
The first author is  supported by China Scholarship Council and the National Natural Science Foundation of China Grant 61472120.
The second author is supported by the Austrian Science Fund FWF Project \red{P 30405-N32}.

 \end{document}